\documentclass[a4paper,11pt]{amsart}
\usepackage{enumitem}   
\usepackage{cite}
\usepackage{tikz-cd}
\usepackage{amsmath,amssymb, amsthm,graphicx,yfonts,mathrsfs,color}
\usepackage[utf8]{inputenc}
\usepackage{braket}
\newcommand{\Addresses}{{
  \bigskip
  \footnotesize

  \textsc{Mathematical Institute, University of Oxford, Woodstock
Rd, Oxford OX2 6GG, UK}\par\nopagebreak
  \textit{E-mail address:} \texttt{francesca.bianchi@maths.ox.ac.uk}
}}

\makeatletter
\DeclareMathOperator{\GL}{GL}
\newcommand{\subalign}[1]{%
  \vcenter{%
    \Let@ \restore@math@cr \default@tag
    \baselineskip\fontdimen10 \scriptfont\tw@
    \advance\baselineskip\fontdimen12 \scriptfont\tw@
    \lineskip\thr@@\fontdimen8 \scriptfont\thr@@
    \lineskiplimit\lineskip
    \ialign{\hfil$\m@th\scriptstyle##$&$\m@th\scriptstyle{}##$\crcr
      #1\crcr
    }%
  }
}
\makeatother

\theoremstyle{plain}
\newtheorem{thm}{Theorem}[section]
\newtheorem{prop}[thm]{Proposition}
\newtheorem{lemma}[thm]{Lemma}

\newtheorem{cor}[thm]{Corollary}
\theoremstyle{definition}
\newtheorem{mydef}[thm]{Definition}
\newtheorem{rmk}[thm]{Remark}

\newtheorem{ass}{Assumption}
\makeatletter
\def\th@plain{%
  \thm@notefont{}
  \itshape 
}
\def\th@definition{%
  \thm@notefont{}
  \normalfont 
}
\makeatother
\author{Francesca Bianchi}
\title{A result on the analytic $\mu$-invariant of elliptic curves}
\date{\today}
\begin{document}
\begin{abstract}
We generalise a result by Greenberg and Vatsal on the relation between the analytic $\mu$-invariants of two elliptic curves whose $p^i$-torsion subgroups are isomorphic as Galois modules, for suitable $i$.
\end{abstract}
\maketitle
\section{Introduction}
\sloppy
Let $E$ be an elliptic curve over the field of rational numbers $\mathbb{Q}$ and let $p$ be an odd prime at which $E$ has good ordinary reduction. Mazur and Swinnerton-Dyer \cite{arithmeticofWeilcurves} attached to the pair $(E,p)$ a $p$-adic $L$-function $\mathcal{L}_p(E/\mathbb{Q},T)$, which lies in $\Lambda\otimes \mathbb{Q}_p$, where $\Lambda=\mathbb{Z}_p[[T]]$ is the ring of power series in the variable $T$ with $p$-adic integral coefficients. The function $\mathcal{L}_p(E/\mathbb{Q},T)$ interpolates the central values of the twists of the complex $L$-function of $E$ by the characters of the absolute Galois group of $\mathbb{Q}$ factoring through the $\mathbb{Z}_p$-cyclotomic extension.

The analytic $\mu$-invariant of $\mathcal{L}_p(E/\mathbb{Q},T)$, denoted $\mu_E$, is defined as the $p$-adic valuation of the highest power of $p$ dividing $\mathcal{L}_p(E/\mathbb{Q},T)$. Wuthrich \cite{wuthrichintegral} proved that $\mathcal{L}_p(E/\mathbb{Q},T)$ is in fact an integral power series, when $p$ is a prime of good ordinary reduction: that is, $\mathcal{L}_p(E/\mathbb{Q},T)\in \Lambda$ and $\mu_E\geq 0$. 

The main result of this article is the following.

\begin{thm}
\label{main}
Let $E_1/\mathbb{Q}$ and $E_2/\mathbb{Q}$ be elliptic curves and $p$ an odd prime at which both $E_1$ and $E_2$ have good ordinary reduction. Assume that $E_1[p^{\mu_{E_1}}]$ and $E_2[p^{\mu_{E_1}}]$ are isomorphic as Galois modules and that $E_1[p]$ and $E_2[p]$ are irreducible. Then $\mu_{E_1}\leq \mu_{E_2}$. Furthermore, if $E_1[p^{\mu_{E_1}+1}]$ and $E_2[p^{\mu_{E_1}+1}]$ are isomorphic, then $\mu_{E_1}=\mu_{E_2}$.
\end{thm}

Theorem \ref{main} is a generalisation of a part of\cite[Theorem (1.4)]{invariants}, which covers the case when $\mu_{E_1}=0$.  By extending Greenberg and Vatsal's result to a more general setting, we aim to provide a perhaps more detailed proof of their statement in \cite{invariants} as well. A key ingredient is the theory of canonical periods of cuspidal eigenforms and resulting congruences between algebraic values of $L$-functions developed by Vatsal in \cite{vatsal} and \cite{vatsalintegralperiods}.

We remark that Theorem \ref{main} already has an algebraic counterpart in the literature. Namely, Barman and Saikia \cite{barman} proved the analogue of Theorem \ref{main} for the \emph{algebraic} $\mu$-invariants of $E_1$ and $E_2$. Here by the algebraic $\mu$-invariant of $(E,p)$ we mean the $p$-adic valuation of the highest power of $p$ dividing the characteristic series of the Pontryagin dual of the $p$-primary part of the Selmer group of $E$ over the $\mathbb{Z}_p$-cyclotomic extension of $\mathbb{Q}$ (see \cite{green}). Barman and Saikia's result is proved under the assumption of triviality of the groups $E_1(\mathbb{Q})[p]$ and $E_2(\mathbb{Q})[p]$ replacing the stronger irreducibility of $E_1[p]$ and $E_2[p]$.

We note that the method we use does extend to the case when $E_1[p]$ and $E_2[p]$ are reducible, as long as the unique unramified character appearing in their semisimplification is odd. However, Greenberg and Vatsal showed in \cite[Theorem (1.3)]{invariants} that in this case the $\mu$-invariant is zero.\\

{\bf Acknowledgements.} The author is extremely grateful to Andrew Wiles for his time and the precious technical help provided. She would also like to thank Jennifer Balakrishnan for her support and for proof-reading and Thanasis Bouganis for some conversations had on the topic.
\section{Notation and auxiliary results}
Let $p$ be an odd prime and fix once and for all embeddings of $\bar{\mathbb{Q}}$ into $\mathbb{C}$ and into the algebraic closure $\bar{\mathbb{Q}}_p$ of $\mathbb{Q}_p$.

We begin by restricting some definitions and results from \cite{vatsal} and \cite{vatsalintegralperiods} to the case of weight 2 cusp forms (rather than dealing with arbitrary weight).
Let $f$ be a normalised cuspidal eigenform of weight $2$ on $\Gamma=\Gamma_1(M)$, for an integer $M\geq 4$. Let $K$ be the $p$-adic field obtained by completing the number field generated over $\mathbb{Q}$ by the coefficients of $f$ at a prime $\mathfrak{p}$ above $p$ and let $O$ be the ring of integers of $K$. Furthermore, denote by $S_2(\Gamma,O)$ the ring of weight $2$ cusp forms on $\Gamma$ with coefficients in $O$ and by $\mathbb{T}$ the $O$-algebra generated by the Hecke operators acting on $S_2(\Gamma, O)$. There is a homomorphism $\lambda_f: \mathbb{T}\to O$, which maps a Hecke operator to the corresponding eigenvalue of $f$. We denote its kernel by $I_f$ and write $\mathfrak{m}$ for the maximal ideal of $\mathbb{T}$ determined by $\mathfrak{p}$ and $I_f$. Note, in particular, that if $g$ is another normalised eigenform in $S_2(\Gamma,O)$ which satisfies a congruence $g\equiv f \mod \mathfrak{p}$, then $f$ and $g$ determine the same maximal ideal $\mathfrak{m}$.

Endow the ring $O$ with the trivial action of the group $\Gamma$. Then the first cohomology group of the $\Gamma$-module $O$ is $H^1(\Gamma,O)=\text{Hom}(\Gamma, O)$. Let $P$ be the set of parabolic elements of $\Gamma$, i.e. the set of matrices in $\Gamma$ which fix exactly one element of $\mathbb{R}\cup \{\infty\}$. The parabolic cohomology group $H^1_P(\Gamma,O)$ of Eichler-Shimura is the subgroup of $H^1(\Gamma,O)=\text{Hom}(\Gamma,O)$ consisting of those homomorphisms which vanish on $P$.

Besides being a group, $H_P^1(\Gamma,O)$ also carries the structure of a $\mathbb{T}$-module (see \cite{shimuraauto} or \cite[p. 116]{diamondim} for an explicit description of the action of double cosets). We denote by $H_P^1(\Gamma,O)_{\mathfrak{m}}$ the localisation of $H_P^1(\Gamma,O)$ at $\mathfrak{m}$.
Furthermore, complex conjugation acts on $H_P^1(\Gamma, O)$ and we write $H_P^1(\Gamma, O)^{\pm}$ for the $\pm$-eigenspaces.

The results in \cite{vatsal} and \cite{vatsalintegralperiods} are then derived under the following assumption.
\begin{ass}
\label{mainassumption}
There is an $\alpha\in \{+,-\}$ such that $H_P^1(\Gamma,O)_{\mathfrak{m}}^{\alpha}$ is cofree of rank $1$, meaning that its dual is a free $\mathbb{T}$-module of rank $1$. Such an $\alpha$ is called an \emph{admissible sign}. In symbols,
\begin{equation*}
H_P^1(\Gamma, O)_{\mathfrak{m}}^{\alpha}\cong \text{Hom}_O(\mathbb{T}_{\mathfrak{m}},O)_{\mathfrak{m}}.
\end{equation*}
\end{ass}
By duality, we also have an isomorphism $S_2(\Gamma,O)\cong \text{Hom}_O(\mathbb{T},O)$ as $\mathbb{T}$-modules (see \cite[Proposition 12.4.13]{diamondim}). Therefore, Assumption \ref{mainassumption} implies that $H_P^1(\Gamma, O)_{\mathfrak{m}}^{\alpha}\cong S_2(\Gamma,O)_{\mathfrak{m}}$ and we may fix an isomorphism $\iota:S_2(\Gamma,O)_{\mathfrak{m}}\cong H_P^1(\Gamma, O)_{\mathfrak{m}}^{\alpha}$. Denote by $\delta_f^{\alpha}$ the element in the parabolic cohomology group corresponding to $f$ under this isomorphism, i.e. $\delta_f^{\alpha}=\iota(f)$.

If $\mathbb{T}_{\mathfrak{m}}\otimes K =\prod R_i$ for local $K$-algebras $R_i$, then there is a unique $i$ such that $\lambda_f$ factors through $R_i$ and we have $K=R_i$ (see \cite{vatsalintegralperiods} for more details).\\

Consider the $\mathbb{C}$-valued differential form $f(z)dz$. The map $\gamma\mapsto \int_{z_0}^{\gamma z_0} f(z)dz$ is a homomorphism $\Gamma\to \mathbb{C}$, which vanishes on the parabolic subgroup $P$ and which is independent of the choice of a basepoint $z_0$ in the upper half plane. Denote by $\omega_f$ the cohomology class in $H_P^1(\Gamma,\mathbb{C})$ corresponding to the $1$-cocycle just constructed. Then $\omega_f$ may be decomposed into a sum of eigenvectors $\omega_f^{+}$ and $\omega_f^{-}$ for the action of complex conjugation. Since $f$ is an eigenform of level $M$, the cocycle $\omega_f^{\pm}$ is an eigenvector also for the Hecke algebra of level $M$. The same is true for $\delta_f^{\alpha}$. Since $K$ is a $1$-dimensional space, there thus exists $\Omega_f^{\alpha}$ such that $\Omega_f^{\alpha}\delta_f^{\alpha}=\omega_f^{\alpha}$. It is clear that $\Omega_f^{\alpha}$ depends on the choice of isomorphism $\iota$; however, periods $\Omega_f^{\alpha}$ corresponding to different isomorphisms differ merely by multiplication by a $p$-adic unit.

\begin{mydef}
\label{canonical}
Let $\Omega_f^{\alpha}\in \mathbb{C}$ be as above. Then $\Omega_f^{\alpha}$ is a \emph{canonical period} attached to $f$ and $\alpha$, unique up to multiplication by a $p$-adic unit.
\end{mydef}
\begin{rmk}
The canonical period of Definition \ref{canonical} corresponds to a choice of transcendental Shimura period. In particular, let $f=\sum_{n=1}^{\infty}a_n q^n$ be a weight $2$ normalised eigenform of level $M$ and let $L(f,s)=\sum_{n=1}^{\infty}a_n n^{-s}$ be the complex $L$-function of $f$, which is convergent for $\text{Re}(s)> \frac{3}{2}$ and admits an analytic continuation to the whole complex plane. If $\chi$ is a Dirichlet character, we also write $L(f,\chi,s)=\sum_{n=1}^{\infty}a_n\chi(n)n^{-s}$ and $\alpha=\chi(-1)$ (note that $L(f,s)=L(f,\mathbb{I},s)$ where $\mathbb{I}$ is the trivial character). Then Shimura \cite{shimuraspecialvalues} proved the existence of certain transcendental periods, denoted $\Omega_f^{Shim,\pm}$, with the property that the ratio $\frac{L(f,\chi, 1)}{(-2\pi i)\Omega_f^{Shim,\alpha}}$ is an algebraic number. However, the quantity $\Omega_f^{Shim,\alpha}$ is only defined up to multiplication by an element in ${\bar{\mathbb{Q}}}^{\times}$. The definition of $\Omega_f^{\alpha}$ then corresponds to a choice of Shimura period (up to a $p$-adic unit).
\end{rmk}

The canonical periods will appear in the normalisation in the interpolation property of the $p$-adic $L$-function of $f$. When $f$ is a weight $2$ newform corresponding to an elliptic curve $E$, the canonical periods of $f$ are related to the N\'eron periods of $E$. In particular, let $E$ be an elliptic curve over $\mathbb{Q}$ with unique global minimal Weierstrass equation
\begin{equation*}
y^2+a_1xy+a_3y=x^3+a_2x^2+a_4x+a_6,
\end{equation*}
with $a_1,a_2,a_3,a_4,a_6\in \mathbb{Z}$, $a_1,a_3\in \{0,1\}$ and $a_2\in \{-1,0,1\}$. The invariant differential of this Weierstrass equation for $E$ is $\omega_E=\frac{dx}{2y+a_1x+a_3}$. Let $H_1(E,\mathbb{Z})$ be the first $\mathbb{Z}$-homology group of $E$. Then $H_1(E,\mathbb{Z})$ decomposes as a direct sum of two copies of $\mathbb{Z}$, each being an eigenspace for the action of complex conjugation. In particular, let $\gamma_{\alpha}$ be a generator of $H_1(E,\mathbb{Z})^{\alpha}$. We define the N\'eron periods of $E$ as $\Omega_E^{\alpha}=\int_{\gamma_{\alpha}}\omega_E$.

\begin{prop}
\label{neronequalf}
Let $E$ be an elliptic curve over $\mathbb{Q}$ and let $f$ be the newform corresponding to $E$ by modularity. Let $p$ be a prime of good ordinary or multiplicative reduction for $E$ and assume that $E[p]$ is irreducible. Then $\Omega_E^{\alpha}$ and $(-2\pi i)\Omega_f^{\alpha}$ are equal up to multiplication by a $p$-adic unit.
\end{prop}
\begin{proof}
This is \cite[Proposition (3.1)]{invariants}. Note that, under the assumption that $E[p]$ is irreducible, the N\'eron periods of the elliptic curves in the isogeny class of $E$ are equal up to multiplication by a $p$-adic unit. Thus, the assumption of optimality may be removed from \emph{loc. cit}.

We briefly sketch the steps in the proof. Denote by $N$ the conductor of $E$. Let $E_0$ (resp. $E_1$) be the $X_0(N)$-optimal (resp. $X_1(N)$-optimal) curve in the isogeny class for $E$ and let $\pi_0:X_0(N)\to E_0$ (resp. $\pi_1:X_1(N)\to E_1$) be a strong (resp. optimal) parametrisation. Then $\pi_1^{*}(\omega_{E_1})=c_1((2\pi i)f(z)dz)$ for some $c_1\in \mathbb{Z}$ (cf. \cite[Theorem 1.6]{stevensstickelberger}). Similarly, $\pi_0^*(\omega_{E_0})=c_0((2\pi i)f(z)dz)$, where $c_0\in \mathbb{Z}$ is also known to be a $p$-adic unit by a result of Mazur \cite{mazurrationalprime}. It is then shown in \cite[Proposition (3.3)]{invariants} that $c_0$ and $c_1$ differ by a $p$-adic unit in $\mathbb{Z}_p$, from which we deduce what $c_1$ is a $p$-adic unit.

Finally, one considers the map $H_1(X_1(N), \mathbb{Z}_p)^{\alpha}\to \mathbb{C}$, which sends a $1$-cycle $\gamma$ to $\int_{\gamma}(2\pi i)f(z)dz=(2\pi i)\Omega_f^{\alpha} (\gamma \cap \delta_f^{\alpha})$ (note that we have $H^1_P(\Gamma,\mathbb{Z}_p)^{\alpha}\cong H^1(X_1(N),\mathbb{Z}_p)^{\alpha}$ \cite[\S 8.1]{shimuraauto} and $\cap$ denotes the cap product). Since $f$ is normalised, it is in particular non-zero modulo $p$ and hence there exists an element $\gamma\in H_1(X_1(N),\mathbb{Z}_p)^{\alpha}$ such that $\gamma \cap \delta_f^{\alpha}$ is a $p$-adic unit. Therefore the image of the integration map defined above is a free $\mathbb{Z}_p$-module of rank $1$, generated by $(2\pi i)\Omega_f^{\alpha}$. We similarly define a map $H^1(E_1,\mathbb{Z}_p)^{\alpha}\to \mathbb{C}$ by $\gamma\to \int_{\gamma}\omega_{E_1}$, whose image is spanned by $\Omega_{E_1}^{\alpha}$. It remains to notice that the images of the two integration maps coincide.
\end{proof}

Let $f=\sum_{n=1}^{\infty}a_nq^n$ be a normalised weight $2$ eigenform on $\Gamma_1(M)$. Say that $f$ is \emph{good ordinary at $p$} if $p\nmid M$ and the polynomial $X^2-a_pX+p$ has a unique root, say $\alpha_p$, of vanishing $p$-adic valuation. Let $\chi$ be a Dirichlet character of conductor $p^k m$, where $k\in\mathbb{Z}_{\geq 0}$ and $m\in \mathbb{Z}_{\geq 1}$, with $p\nmid m$. We define the $p$-adic multiplier of the pair $(f,\chi)$ as
\begin{equation*}
e_p(\chi)=\frac{1}{\alpha_p^k}\bigg(1-\frac{\bar{\psi}(p)}{\alpha_p}\bigg)\bigg(1-\frac{\psi(p)}{\alpha_p}\bigg)
\end{equation*} 
and denote by $\tau(\chi)$ and $\mathfrak{f}_{\chi}$ the Gauss sum and the conductor of $\chi$.

Let $\gamma^{cyc}$ be a topological generator of the Galois group of the $\mathbb{Z}_p$-cyclotomic extension of $\mathbb{Q}$.
\begin{mydef}[The $p$-adic $L$-function of a weight $2$ eigenform]
\label{def p-adic L function}
Let $f$ be a normalised cuspidal eigenform of weight $2$ on $\Gamma_1(M)$, for some integer $M$, and assume that $f$ is good ordinary at the odd prime $p$. Let $\chi$ be a Dirichlet character of sign $\alpha$ and $\Lambda=\mathbb{Z}_p[\chi][[T]]$. Assume that $\alpha$ is an admissible sign (cf. Assumption \ref{mainassumption}). The \emph{$p$-adic $L$-function of $f$ twisted by $\chi$} is the unique element $\mathcal{L}_p(f,\chi,T)\in\Lambda\otimes \mathbb{Q}_p$ satisfying the interpolation property
\begin{equation*}
\mathcal{L}_p(f,\chi,\zeta-1)=e_p(\chi\rho)\frac{\mathfrak{f}_{\chi\rho}}{\tau(\chi^{-1}\rho^{-1})}\cdot \frac{L(f,\chi^{-1}\rho^{-1},1)}{(-2\pi i)\Omega_f^{\alpha}},
\end{equation*}
for every root of unity of $p$-power order $\zeta$ and where $\rho$ is the Dirichlet character corresponding to the Galois character which maps $\gamma^{cyc}$ to $\zeta$.

Similarly, for a finite set of primes $\Sigma$ not containing $p$ we define the \emph{$\Sigma$-incomplete $p$-adic $L$-function of $f$ twisted by $\chi$} as the unique element $\mathcal{L}_p^{\Sigma}(f,\chi,T)\in\Lambda\otimes \mathbb{Q}_p$ satisfying the interpolation property
\begin{equation*}
\mathcal{L}_p^{\Sigma}(f,\chi,\zeta-1)=e_p(\chi\rho)\frac{\mathfrak{f}_{\chi\rho}}{\tau(\chi^{-1}\rho^{-1})}\cdot \frac{L^{\Sigma}(f,\chi^{-1}\rho^{-1},1)}{(-2\pi i)\Omega_f^{\alpha}},
\end{equation*}
where $\zeta$ and $\rho$ are as above and $L^{\Sigma}(f,\chi^{-1}\rho^{-1},s)$ is the $L$-function of $f$ twisted by $\chi^{-1}\rho^{-1}$ with the Euler factors at the primes in $\Sigma$ removed.

If $E$ is an elliptic curve over $\mathbb{Q}$ and $f$ is the corresponding newform, we write $\mathcal{L}_p(E,\chi,T)$ for $\frac{(-2\pi i)\Omega_f^{\alpha}}{\Omega_E^{\alpha}}\cdot \mathcal{L}_p(f,\chi,T)$ and $\mathcal{L}_p^{\Sigma}(E,\chi,T)$ for $\frac{(-2\pi i)\Omega_f^{\alpha}}{\Omega_E^{\alpha}}\cdot \mathcal{L}_p^{\Sigma}(f,\chi,T)$.
\end{mydef}
Let $f$, $\chi$ be as in Definition \ref{def p-adic L function}; let $\pi$ be a uniformiser for $\mathbb{Z}_p[\chi]$ and denote by $\nu_p$ the $p$-adic valuation on $\mathbb{Z}_p[\chi]$, normalised so that $\nu_p(p)=1$. We define the analytic $\mu$-invariant $\mu_{f,\chi}$ of $\mathcal{L}_p(f,\chi,T)$ as the $p$-adic valuation of the highest power of $\pi$ dividing  $\mathcal{L}_p(f,\chi,T)$ if $\mathcal{L}_p(f,\chi,T)\in\Lambda$ and as $-\min_{a}\{\nu_p(\pi^a):\pi^a\mathcal{L}_p(f,\chi,T)\in \Lambda\}$ otherwise. Given a finite set of primes $\Sigma\not\ni p$ we also define the analytic $\mu$-invariant of $\mathcal{L}_p^{\Sigma}(f,\chi,T)$, denoted $\mu_{f,\chi,\Sigma}$, in a similar way.

\begin{lemma}
\label{muinvariantsareequal}
Let $f,\chi,\Sigma$ be as in Definition \ref{def p-adic L function}. Then $\mu_{f,\chi}=\mu_{f,\chi,\Sigma}$.
\end{lemma}
\begin{proof}
This is a slight generalisation of \cite[pp. 24-25]{invariants}, where the statement is proved for trivial character $\chi$. As in \cite{invariants}, for a prime $\ell\in \Sigma$, write $[(1-\alpha_{\ell}\ell^{-s})(1-\beta_{\ell}\ell^{-s})]^{-1}$ for the Euler factor at $\ell$ of the $L$-function $L(f,s)$. Then we have $\mathcal{L}_p^{\Sigma}(f,\chi,T)=\mathcal{L}_p(f,\chi,T)\cdot \prod_{\ell \in \Sigma}\mathcal{P}_{\ell}(T)$, where $\mathcal{P}_{\ell}(T)=(1-\alpha_{\ell}\chi^{-1}(\ell)\ell^{-1}(1+T)^{f_{\ell}})(1-\beta_{\ell}\chi^{-1}(\ell)\ell^{-1}(1+T)^{f_{\ell}})$ and where $f_{\ell}\in \mathbb{Z}_p$ is such that $(\gamma^{cyc})^{-f_{\ell}}$ is the Frobenius automorphism for $\ell$ in $\Gamma$. Moreover, $\mathcal{P}_{\ell}(T)\in \Lambda\setminus p\Lambda$ for all $\ell\in \Sigma$: if $\ell|\mathfrak{f}_{\chi}$, then it is clear; else $\chi^{-1}(\ell)$ is a $p$-adic unit and the argument is the same as for trivial character. 
\end{proof}

\begin{lemma}
\label{canonicalperiodsremovingEulerfactors}
Let $E$ be an elliptic curve of level $N$, let $f$ the corresponding newform and let $p$ be an odd prime. Suppose $f$ is good ordinary at $p$ and let $\alpha$ be an admissible sign for $f$. Let $g$ be an eigenform obtained from $f$ by deleting all Euler factors corresponding to the primes in a finite set $\Sigma\not\ni p$ and assume that $\alpha$ is also admissible for $g$. Then the canonical periods $\Omega_f^{\alpha}$ and $\Omega_g^{\alpha}$ are equal up to multiplication by a $p$-adic unit. 
\end{lemma}
\begin{proof}
See \cite[\S 4.1]{vatsalintegralperiods}. Note that the statement there is proved under the assumption that $f$ and $g$ are $p$-stabilised, but the argument also applies here, as we are dealing with forms of weight $2$, where one can remove the assumption of $p$-stabilisation (cf. \cite[Lemma 4.1]{vatsalintegralperiods}).
\end{proof}
We next state a result from \cite{vatsalintegralperiods} (Theorem 3.10) which we will need to establish a congruence between $p$-adic $L$-functions.
\begin{prop}
\label{mainvatsal}
Let $f$ and $g$ be normalised cuspidal eigenforms of weight $2$ and of the same level, say $M$, such that $K=\mathbb{Q}(f)=\mathbb{Q}(g)$. Suppose that, for some integer $r\geq 1$, $f$ and $g$ satisfy the congruence $f\equiv g \mod \mathfrak{p}^r$, where $\mathfrak{p}$ is a prime of $K$ above $p$ and that $\alpha$ is an admissible sign for $f$ (and hence for $g$). Then, for every Dirichlet character $\chi$ of sign $\alpha$, we have
\begin{equation*}
\frac{\mathfrak{f}_{\chi}}{\tau(\chi^{-1})}\cdot \frac{L(f,\chi^{-1},1)}{(-2\pi i)\Omega_f^{\alpha}}\equiv \frac{\mathfrak{f}_{\chi}}{\tau(\chi^{-1})}\cdot \frac{L(g,\chi^{-1},1)}{(-2\pi i)\Omega_g^{\alpha}} \mod \mathfrak{p}^r.
\end{equation*}
\end{prop}

\begin{cor}
\label{congruencepadic}
With the assumptions of Proposition \ref{mainvatsal}, suppose further that $f$ and $g$ are good ordinary at $p$ and denote by $O$ the ring of integers of $K$. Then
\begin{equation*}
\mathcal{L}_p(f,\chi,\zeta-1) \equiv \mathcal{L}_p(g,\chi,\zeta-1) \mod \mathfrak{p}^r
\end{equation*}
for all $p$-power roots of unity $\zeta$. Therefore
\begin{equation*}
\mathcal{L}_p(f,\chi,T)\equiv \mathcal{L}_p(g,\chi,T) \mod \mathfrak{p}^rO(\chi)[[T]].
\end{equation*}
\end{cor}
\begin{proof}
Let $a_p$ respectively $b_p$ be the eigenvalue of $f$ resp. $g$ for the Hecke operator $T_p$. Since $f\equiv g \mod \mathfrak{p}^r$, in particular we know that $a_p\equiv b_p \mod \mathfrak{p}^r$. By Hensel's lemma, this implies that the unique $p$-adic unit root of $X^2-a_pX+p$ is congruent modulo $\mathfrak{p}^r$ to the unique $p$-adic unit root of $X^2-b_pX+p$, which, together with the expression for $\mathcal{L}_p(f,\chi,\zeta-1)$ from Definition \ref{def p-adic L function} and with Proposition \ref{mainvatsal}, implies the first congruence. \\
As for the second part of the proof, see \cite[Theorem (1.10)]{vatsal}. 
\end{proof}
For a normalised cuspidal eigenform $f$ with rational coefficients, let $\bar{\rho}_f:\text{Gal}(\bar{\mathbb{Q}}/\mathbb{Q})\to \GL_2(\mathbb{F}_p)$ be the unique semisimple representation satisfying $\text{Tr}(\bar{\rho_f}(\text{Frob}_q))\equiv a_q \mod p$ for all $q\nmid Np$. We say that $f$ is \emph{residually irreducible} if $\bar{\rho}_f$ is irreducible.
\begin{lemma}
\label{corollaryeulerremoved2}
Let $f$ be a newform of weight $2$ on $\Gamma_1(N)$  with rational coefficients, $p$ an odd prime such that $p\nmid N$ and suppose that the representation attached to $f$ is residually irreducible. Let $S$ be a (possibly empty) finite set of primes not containing $p$ and let $g_S$ be the eigenform obtained from $f$ by removing the Euler factors of $f$ corresponding to $S$. Then Assumption \ref{mainassumption} holds for $g_S$ for both $\alpha=1$ and $\alpha=-1$.
\end{lemma}
\begin{proof}
We have $\bar{\rho}_{g_S}\cong \bar{\rho}_f$ is irreducible. Furthermore, the level of $g_S$ is not divisible by $p$ so the claim follows directly from \cite[Theorem 2.1 (i)]{wilesmodularity}.
\end{proof}

\section{Proof of the main result}
The theorem is a corollary of the following generalisation of \cite[Theorem (3.10)]{invariants}.
\begin{thm}
\label{conggen}
Let $E_1$ and $E_2$ be elliptic curves of level $N_1$ and $N_2$ and let $i$ be a positive integer such that $E_1[p^i]\cong E_2[p^i]$ as Galois modules and $E_j[p]$ is irreducible for $j=1,2$. Assume that $E_j$ has good ordinary reduction at $p$. Denote by $\Sigma$ the set of primes dividing $N_1N_2$. Then
\begin{equation*}
\mathcal{L}^{\Sigma}(E_1,\chi,T)\equiv u\cdot \mathcal{L}^{\Sigma}(E_2,\chi,T)\ (\text{mod}\ p^i\Lambda) 
\end{equation*}
for every Dirichlet character $\chi$ and some $u\in \mathbb{Z}_p[\chi]^*$. Here $\Lambda=\mathbb{Z}_p[\chi][[T]]$.
\end{thm}

\begin{proof}
Denote by $M$ the least common multiple of $N_1$ and $N_2$ and by $f_1=\sum_{n=1}^{\infty} a_nq^n$ and $f_2=\sum_{n=1}^{\infty} b_n q^n$ the newforms corresponding to $E_1$ and $E_2$ by modularity. Furthermore, let $\rho_{E_1,p^i}$ denote the Galois representation attached to $E_1[p^i]$ and similarly for $E_2$. For all primes $\ell$ at which $\rho_{E_1,p^i}$ resp. $\rho_{E_2,p^i}$ is unramified, we have
\begin{equation*}
\text{tr}(\rho_{E_1,p^i}(\text{Frob}_{\ell}))\equiv a_{\ell} \text{ mod }p^i 
\end{equation*}
resp.
\begin{equation*}
\text{tr}(\rho_{E_2,p^i}(\text{Frob}_{\ell}))\equiv b_{\ell} \text{ mod }p^i 
\end{equation*}
(see e.g. \cite[\S 5]{ribet}). In particular, this holds at all $\ell$ coprime with $Mp$. Since the trace of a matrix is an isomorphism-invariant of representations, the hypothesis that $E_1[p^i]\cong E_2[p^i]$, together with the above, implies that $a_{n}\equiv b_n \text{ mod }p^i$ for all $n$ coprime with $Mp$.

Denote by $D_p$ the decomposition group at $p$ and by $\rho_{E_j}$ the representation of $\text{Gal}(\bar{\mathbb{Q}}/\mathbb{Q})$ on the $p$-adic Tate module of $E_j$. Since $E_1$ and $E_2$ are ordinary at $p$ by assumption, by \cite[Theorem 2]{wiles} we have
\begin{equation*}
\rho_{E_1}|_{D_p} \sim \left(\begin{smallmatrix}\epsilon_1^1 & *^1\\
0 & \epsilon_2^1 \end{smallmatrix}\right) \qquad \rho_{E_2}|_{D_p}\sim \left(\begin{smallmatrix}\epsilon_1^2 & *^2\\
0 & \epsilon_2^2 \end{smallmatrix}\right),
\end{equation*}
where $\epsilon_2^1$ and $\epsilon_2^2$ are unramified and $\epsilon_2^1(\text{Frob}_p)=\alpha_{E_1}$ and $\epsilon_2^2(\text{Frob}_p)=\alpha_{E_2}$. Here $\alpha_{E_1}$ is the unique unit root of $X^2-a_pX+p$ and $\alpha_{E_2}$ is the unique unit root of $X^2-b_pX+p$. Since $\rho_{E_j,p^i}\cong \rho_{E_j} \mod p^i$, we conclude that $\alpha_{E_1}\equiv \alpha_{E_2} \mod p^i$.

Let $\beta_{E_1}$ respectively $\beta_{E_2}$ be the other root of the corresponding characteristic polynomials of Frobenius. Then $\alpha_{E_1}\beta_{E_1}=p=\alpha_{E_2}\beta_{E_2}$. Since $\alpha_{E_1}\equiv \alpha_{E_2} \mod p^i$ and $\nu_p(\alpha_{E_j})=0$, we also have the congruence $\beta_{E_1}\equiv\beta_{E_2} \mod p^i$. Therefore $a_p=-(\alpha_{E_1}+\beta_{E_1})\equiv -(\alpha_{E_2}+\beta_{E_2})  \equiv b_p$.\\
This shows that
\begin{equation*}
a_n\equiv b_n \text{ mod }p^i,\ \text{for all}\ (n,M)=1.
\end{equation*}
Let now $g_1$ and $g_2$ be the normalised cuspidal eigenforms of weight 2 obtained from $f_1$ and $f_2$ by deleting all Euler factors at primes in $\Sigma$. These may be viewed as eigenforms of level $M^2$. By construction, $g_1$ and $g_2$ satisfy the congruence $g_1\equiv g_2 \mod p^i$.

We wish to apply Corollary \ref{congruencepadic}. In order to do so, we first need to check that $\alpha$ is an admissible sign for $g_1$ and $g_2$. But this follows from the assumption that $E_1[p]$ and $E_2[p]$ are irreducible together with Lemma \ref{corollaryeulerremoved2}.
Clearly $g_1$ and $g_2$ are good ordinary at $p$ and, with the notation of Corollary \ref{congruencepadic}, we have $K=\mathbb{Q}$. Therefore, we have a congruence
\begin{equation*}
\mathcal{L}_p(g_1,\chi,T)\equiv \mathcal{L}_p(g_2,\chi,T) \mod p^i \mathbb{Z}_p(\chi)[[T]].
\end{equation*}
Finally, by Lemma \ref{canonicalperiodsremovingEulerfactors}, the canonical periods of $f_1$ and $g_1$ differ by a $p$-adic unit. Moreover,
$L(g_1,\chi,s)=L^{\Sigma}(f_1,\chi,s)$. Similarly for $f_2$ and $g_2$. This, together with Proposition \ref{neronequalf}, completes the proof.
\end{proof}
We are now ready to prove Theorem \ref{main}.
\begin{proof}[Proof of Theorem \ref{main}]
If $\mu_{E_1}=0$, the inequality $\mu_{E_1}\leq \mu_{E_2}$ holds trivially. If $\mu_{E_1}>0$, applying Theorem \ref{conggen} with $i=\mu_{E_1}$ and $\chi$ the trivial character, we obtain the congruence
\begin{equation*}
\mathcal{L}^{\Sigma}(E_1,T)\equiv u\cdot \mathcal{L}^{\Sigma}(E_2,T) \mod p^{\mu_{E_1}}\mathbb{Z}_p[[T]],
\end{equation*}
for some $u\in \mathbb{Z}_p^*$. By Lemma \ref{muinvariantsareequal}, $p^{\mu_{E_1}}$ exactly divides $\mathcal{L}^{\Sigma}(E_1,T)$ and hence $p^{\mu_{E_1}}$ divides $\mathcal{L}^{\Sigma}(E_2,T)$. Finally, Lemma \ref{muinvariantsareequal} applied to $E_2$ then gives the inequality.
Similarly, Theorem \ref{conggen} with $i=\mu_{E_1}+1$ gives
\begin{equation*}
\frac{\mathcal{L}^{\Sigma}(E_1,T)}{p^{\mu_{E_1}}}\equiv u\cdot \frac{\mathcal{L}^{\Sigma}(E_2,T)}{p^{\mu_{E_1}}}\ (\text{mod}\ p\Lambda)
\end{equation*} 
and since the left hand side is a unit, so is the right hand side.
\end{proof}
The assumption of the irreducibility of $E_1[p]$ and $E_2[p]$ is rather strong. In fact, Greenberg conjectured that in this case the algebraic $\mu$-invariants of $E_1$ and $E_2$ should vanish (see \cite[Conjecture 1.11]{green}). Thus the most interesting case of Theorem \ref{main} is still probably the one covered in \cite{invariants}. However, Theorem \ref{conggen} also implies a stronger version of Theorem \ref{main}, as follows.
\begin{thm}
Let $E_1/\mathbb{Q}$ and $E_2/\mathbb{Q}$ be elliptic curves, $p$ an odd prime at which both $E_1$ and $E_2$ have good ordinary reduction and $\chi$ a Dirichlet character. Assume that $E_1[p^{\mu_{E_1,\chi}}]$ and $E_2[p^{\mu_{E_1,\chi}}]$ are isomorphic as Galois modules and that $E_1[p]$ and $E_2[p]$ are irreducible. Then $\mu_{E_1,\chi}\leq \mu_{E_2,\chi}$. Furthermore, if $E_1[p^{\mu_{E_1,\chi}+1}]$ and $E_2[p^{\mu_{E_1,\chi}+1}]$ are isomorphic, then $\mu_{E_1,\chi}=\mu_{E_2,\chi}$.
\end{thm}
\begin{proof}
The proof follows from Theorem \ref{conggen} in the same way as Theorem \ref{main}. We remark that the congruence of Theorem \ref{main} clearly holds also modulo $\pi^i\Lambda$, for a uniformiser $\pi$ of $\mathbb{Z}_p[\chi]$.
\end{proof}
We also note that the statement of Theorem \ref{main} remains true if the assumption of good ordinary reduction at the prime $p$ is replaced by that of multiplicative reduction. In this case, one appeals to \cite[Theorem 2.1 (ii)]{wilesmodularity} to prove that both the newforms $f_1,f_2$ attached to $E_1$ and $E_2$ and their modified forms $g_1,g_2$, obtained by deleting all Euler factors at primes $q\neq p$ dividing $N_1N_2$, satisfy the cofreeness assumption (Assumption \ref{mainassumption}). With an appropriate modification of Definition \ref{def p-adic L function}, one then obtains a congruence between the $p$-adic $L$-functions of $g_1$ and $g_2$. Finally, an analogue of Lemma \ref{canonicalperiodsremovingEulerfactors} also holds in this case. We leave the details to the reader.
\bibliographystyle{alpha}
\bibliography{bibliomu}

\def\cprime{$'$}
\begin{thebibliography}{MSD74}

\bibitem[BS10]{barman}
R.~Barman and A.~Saikia.
\newblock A note on {I}wasawa {$\mu$}-invariants of elliptic curves.
\newblock {\em Bull. Braz. Math. Soc. (N.S.)}, 41(3):399--407, 2010.

\bibitem[DI95]{diamondim}
F.~Diamond and J.~Im.
\newblock Modular forms and modular curves.
\newblock In {\em Seminar on {F}ermat's {L}ast {T}heorem ({T}oronto, {ON},
  1993--1994)}, volume~17 of {\em CMS Conf. Proc.}, pages 39--133. Amer. Math.
  Soc., Providence, RI, 1995.

\bibitem[Gre99]{green}
R.~Greenberg.
\newblock Iwasawa theory for elliptic curves.
\newblock In {\em Arithmetic theory of elliptic curves ({C}etraro, 1997)},
  volume 1716 of {\em Lecture Notes in Math.}, pages 51--144. Springer, Berlin,
  1999.

\bibitem[GV00]{invariants}
R.~Greenberg and V.~Vatsal.
\newblock On the {I}wasawa invariants of elliptic curves.
\newblock {\em Invent. Math.}, 142(1):17--63, 2000.

\bibitem[Maz78]{mazurrationalprime}
B.~Mazur.
\newblock Rational isogenies of prime degree (with an appendix by {D}.
  {G}oldfeld).
\newblock {\em Invent. Math.}, 44(2):129--162, 1978.

\bibitem[MSD74]{arithmeticofWeilcurves}
B.~Mazur and P.~Swinnerton-Dyer.
\newblock Arithmetic of {W}eil curves.
\newblock {\em Invent. Math.}, 25:1--61, 1974.

\bibitem[Rib95]{ribet}
K.~A. Ribet.
\newblock Galois representations and modular forms.
\newblock {\em Bull. Amer. Math. Soc. (N.S.)}, 32(4):375--402, 1995.

\bibitem[Shi76]{shimuraspecialvalues}
G.~Shimura.
\newblock The special values of the zeta functions associated with cusp forms.
\newblock {\em Comm. Pure Appl. Math.}, 29(6):783--804, 1976.

\bibitem[Shi94]{shimuraauto}
G.~Shimura.
\newblock {\em Introduction to the arithmetic theory of automorphic functions},
  volume~11 of {\em Publications of the Mathematical Society of Japan}.
\newblock Princeton University Press, Princeton, NJ, 1994.
\newblock Reprint of the 1971 original, Kan{\^o} Memorial Lectures, 1.

\bibitem[Ste89]{stevensstickelberger}
G.~Stevens.
\newblock Stickelberger elements and modular parametrizations of elliptic
  curves.
\newblock {\em Invent. Math.}, 98(1):75--106, 1989.

\bibitem[Vat99]{vatsal}
V.~Vatsal.
\newblock Canonical periods and congruence formulae.
\newblock {\em Duke Math. J.}, 98(2):397--419, 1999.

\bibitem[Vat13]{vatsalintegralperiods}
V.~Vatsal.
\newblock Integral periods for modular forms.
\newblock {\em Ann. Math. Qu\'e.}, 37(1):109--128, 2013.

\bibitem[Wil88]{wiles}
A.~Wiles.
\newblock On ordinary {$\lambda$}-adic representations associated to modular
  forms.
\newblock {\em Invent. Math.}, 94(3):529--573, 1988.

\bibitem[Wil95]{wilesmodularity}
A.~Wiles.
\newblock Modular elliptic curves and {F}ermat's last theorem.
\newblock {\em Ann. of Math. (2)}, 141(3):443--551, 1995.

\bibitem[Wut14]{wuthrichintegral}
C.~Wuthrich.
\newblock On the integrality of modular symbols and {K}ato's {E}uler system for
  elliptic curves.
\newblock {\em Doc. Math.}, 19:381--402, 2014.

\end{thebibliography}
\Addresses
\end{document}